\documentclass{amsart}
\usepackage{amsmath,amssymb,amsfonts,amsxtra}
\theoremstyle{plain}
\newtheorem{theorem}{Theorem}[section]
\newtheorem{lemma}[theorem]{Lemma}
\newtheorem{corollary}[theorem]{Corollary}
\newtheorem{proposition}[theorem]{Proposition}
\theoremstyle{definition}

\newtheorem{definition}[theorem]{Definition}

\numberwithin{equation}{section}

\def\cc{{\mathcal C}}

\def\bc{{\mathbb C}}

\def\bn{{\mathbb N}}

\def\br{{\mathbb R}}

\def\a{\alpha}

\def\b{\beta}
\def\d{\delta}
\def\e{\varepsilon}

\def\l{\lambda}

\def\r{\rho}

\begin{document}
\title[Factorizing a Schur product]{Decompositions of block Schur products}
\author[Erik Christensen]{Erik Christensen}
\address{Erik Christensen, Institute for Mathematics, University of Copenhagen, Copenhagen, DK-2100, Denmark}
\email{echris@math.ku.dk}
\begin{abstract}
Given two $m \times n $  matrices $A = (a_{ij})$ and $B=(b_{ij}) $ with entries in $B(H)$ for some Hilbert space $H,$ the Schur block product is the  $m \times n$ matrix $ A\square B := (a_{ij}b_{ij}).$ There exists an $m \times n$ matrix $S = (s_{ij})$ with entries from $B(H)$ such that $S$ is a contraction  operator and 
$$ A \square B  = \big(\mathrm{diag}(AA^*)\big)^{\frac{1}{2}}S\big(\mathrm{diag}(B^*B)\big)^{\frac{1}{2}}.$$

The analogus result for the block Schur tensor product $\boxtimes$ defined by Horn and Mathias in \cite{HM} holds too.
 This kind of decomposition of the Schur product seems to be unknown, even for scalar matrices.

Based on the theory of random matrices we show that the set of contractions  $S,$ which may appear in such a  decomposition, is a {\em thin} set in the ball of all contractions.   
\end{abstract}

\maketitle

\section*{INTRODUCTION}

A substantial part of this note is a direct consequence of our previous work on the   block Schur product of matrices with operator entries presented in \cite{Cs}. Unfortunately it has taken us several months to realize that the results we present in \cite{Cs} do have some implications, which are quite easy to obtain, but - hopefully -   interesting. 
On the other hand this article also deals with yet another version of a block Schur product, which was introduced by Horn and Mathias in \cite{HM}. In this block product they replace the ordinary matrix products of the entries in the block matrices  by their spatial tensor products. It turns out that the results on block Schur products, we have obtained lately, do extend to hold in the setting of this block Schur tensor product. A  Schur multiplier is a scalar matrix acting on a block matrix with operator entries, by entry wise products, and we show that the action of a Schur multiplier may be expressed as a a block Schur tensor product, and in this way the results we have obtained are valid in the setting of Schur multipliers too. 

The ordinary Schur product between scalar matrices $R$ and $C$ is denoted $R \circ C.$ The decomposition formula for block Schur products raises a natural question as to how big is the ultra weakly closed convex hull $\cc$ of the set $\{R \circ C\, : R, C \, \in M_n(\bc), \|R\|_r \leq 1 \,\, \|C\|_c \leq 1\,\}$. It turns out that for a given $n$ there exists a self-adjoint operator $Y$ in $M_n(\bc)$ such that $\|Y \| \leq 3 n^{-(1/2)},$ and $Y$ is not in $\cc,$ so in this way we say that  $\cc$ is a thin set. The result is a purely existence result based on the theory of random matrices. This last result was expected from experiments made in {\em Maple } with random matrices, and then it turned out that the theory of random matrices actually can explain what the experiments indicated.  
\section{ON THE BLOCK SCHUR TENSOR PRODUCT}
We emailed our first draft of this note to Roger A. Horn and he responded very quickly and directed  us to  references in \cite{HJ} and  \cite{HM}.  In this section we will focus on the kind of Schur product between block matrices of bounded operators which Horn and Mathias introduces in \cite{HM}. This block product is different from the one Livshits worked on in \cite{Li}, and we recently studied in \cite{Cs}.  
The major difference may in a short way be described via a pair of matrices $A$ and $B$ in $M_n\big(M_k(\bc)\big).$   For these matrices Horn and Mathias defines, what we call, the {\em block Schur tensor product}  $A\boxtimes B $ as the matrix  in $M_n\big(M_k(\bc) \otimes M_k(\bc)\big)$ given by  $(A\boxtimes B)_{ij} := A_{ij} \otimes B_{ij}.$
Livshits defines $A\square B$ as the matrix in $M_n\big(M_k(\bc)\big)$ by $(A\square B)_{ij} : = A_{ij}B_{ij} .$ 
In the scalar case, when $ k = 1$ there is only a notational difference and we get the ordinary Schur product in $M_n(\bc).$ We will now present a formal definition of the block Schur tensor  product $\boxtimes$ given by Horn and Mathias in \cite{HM}, but in the setting of possibly infinite dimensional Hilbert spaces. 

\begin{definition}
Let $I$ and $J$ be index sets, $(E_j)_{j \in J}, \, (K_j)_{j \in J}, \, (F_i)_{i \in I}, \, (L_i)_{i \in I}$ families of Hilbert spaces, $A:= (A_{ij})_{\big((i,j) \in I \times J\big)}, $ $A_{ij} \in B(E_j, F_i)$
\newline and $B:= (B_{ij})_{\big((i,j) \in I \times J\big)} ,$ $B_{ij} \in B(K_j, L_i)$ matrices of bounded operators. The block Schur tensor product $A\boxtimes B $ is defined as the matrix $(A_{ij} \otimes B_{ij})_{\big((i,j) \in I \times J\big)} $ with entries in $B(E_j \otimes K_j, F_i \otimes   L_i).$
\end{definition}

\section{MOTIVATING EXAMPLE}
A scalar Schur multiplier is an $I \times J$  matrix $(s_{ij})$  with scalar entries, and it may  act upon any $I \times J$ matrix $W = (w_{ij}) $ with entries $w_{ij} $ in  vector spaces $V_{ij}$ such that the result is  the matrix $(s_{ij}w_{ij}).$  Such an action of a Schur multiplier  appears in a lot of different settings as for instance in the Example 4.1 in \cite{Ho}, and in the essential parts of the article \cite{Ci}, which was the reason for our interest in Livshits' inequality.  
 The article \cite{Ci}, in turn, has its origin in a problem coming from noncommutative geometry. There one studies  - on a separable Hilbert space - the commutator between a bounded operator $ b$  and a self-adjoint unbounded operator $D$ with compact resolvents. Such a $D$ has a countable real   point spectrum $(\l_i)_{(i \in \bn)}$ and finite dimensional spectral projections $E_i := E(\{\l_i\}).$ When $ b$ is decomposed into the block matrix $b_{ij} := E_i b | E_jH \, \in \, B(E_j, E_i),$ it follows from a simple calculation, that if the commutator $ Db -bD$ is defined and bounded on the domain of definition for $D,$ then it may be decomposed into a block matrix with elements in $B(E_jH, E_iH),$ and the commutator may be described as the Schur product between the scalar matrix $S = (s_{ij} )$ with $s_{ij} = \l_i - \l_j$ and the block matrix $(b_{ij}).$

We will now discus the action of a scalar Schur multiplier on a matrix of bounded operators, and see that this operation may be  be described as a block Schur tensor product. 
Given index sets $I, J,$ Hilbert spaces $(K_j)_{j \in J}, $
$(L_i)_{i \in I,}$  a matrix $B = (b_{ij})$ with entries $b_{ij} \in B(K_j,L_i)$ and a scalar $I \times J$ matrix $S = (s_{ij}).$ In order to obtain the setting of the block Schur tensor product we define Hilbert spaces $E_j := \bc $ and $F_i := \bc$ and identify $s_{ij} $ in the canonical way with an  operator in $B(E_j,F_i).$ The element $( S \boxtimes B)_{ij}  $ acts  according to the definition as follows 
$$ \forall z \in E_j \forall \xi \in K_j: \quad (S\boxtimes B)_{ij} (z \otimes \xi) = s_{ij}z \otimes B_{ij}\xi$$ 
For any Hilbert space $G,$  it is well known that the mapping $z \otimes  \xi \to z\xi$ induces  a linear isometric isomorphism between the Hilbert spaces $\bc \otimes G$ and $G,$ so we see that that {\em the action of the scalar Schur multiplier $S$ is isometrically isomorphic to the action of the matrix $S$ in a block Schur tensor product. }  

\section{LIVSHITS' INEQUALITY FOR THE BLOCK SCHUR TENSOR PRODUCT}

We will not reintroduce the notation from \cite{Cs}, so we will use the words column and row norms right away. In this section  we will fix two sets of indices $I, J,$ they may be finite or infinite, and their elements are denoted $i$ and $j$ respectively. 
Livshits' inequality says 

\begin{theorem}
Let $H$ be a Hilbert space,
 $A=(a_{ij})$ a row bounded matrix
  with elements in $B(H).$  and $B= (b_{ij} ) $ a column bounded matrix
   with elements in $B(H)$ then their block Schur product $A \square B := (a_{ij}b_{ij} ) $ is the matrix of a bounded operator and $\|A \square B\|_{op} \leq \|A\|_r\|B\|_c.$ 
\end{theorem}

This inequality appears in a weaker version already in Schur's original work \cite{Sc}. In satz II, on page 8, where Schur considers two infinite scalar matrices over $\bn$ given as $(u_{pq} ) $ and $(v_{pq} ) $ such that for both matrices both the column norm and the row norm are finite. Then their Schur product is a bounded operator and we may formulate Schur's result as 

$$ \|(u_{pq}v_{pq})\|_{op} \leq \big( \max\{\|u_{pq}\|_c,\|u_{pq}\|_r \}\big) \big( \max\{\|v_{pq}\|_c,\|v_{pq}\|_r \}\big)  .$$ 

In the email from Horn, which was mentioned above, he directed us to Theorem 5.5.3 in \cite{HJ}, and this theorem presents Livshits' inequality in the scalar case, and moreover its proof is nearly identical  to the elementary proof, we present in Section 3 of \cite{Cs}, of  statement (v) in Theorem 2.9  of \cite{Cs}. This statement is - on the other hand - a bit stronger than Livshits' inequality.  
We will now present the natural  extension of this statement to the setting of block Schur tensor products. Again we will rely on the notation of \cite{Cs} even though there are some obvious diferences between matrices $A = (a_{ij}) $ with elements in $B(E_j,F_i) $ and those with elements $a_{ij}$  in $B(H)$ for a fixed $H.$ It is easy to see that in both cases we get, if  such a matrix $A$ is row bounded, then the matrix  $AA^*$ is well defined, as a matrix with bounded operators $(AA^*)_{ij}$ as enttries satifying $\sup\{\|(AA^*)_{ij}\|\} = \|A\|^2_r.$ The main   diagonal diag$(AA^*)$ is a positive bounded diagonal operator of norm $\|A\|^2_r.$ 
 If $B$ is a column bounded matrix with operator entries, then $B^*B$ is a matrix with bounded operator entries and bounded positive  main diagonal diag$(B^*B),$  such that $\|$diag$(B^*B)\| = \|B\|^2_c.$ We will let $1_L$ denote the $I \times I$ matrix,   which is the unit on $\oplus_{i \in I} L_i,$  i.e. a diagonal matrix with the units on $L_i$ in the diagonal. Similarly we let $1_E$ denote the diagonal matrix on $\oplus_{j \in J} E_j $ with the unit on $E_j$ in the $j$th diagonal entry. The operators $1_L$ and $1_E$ are of course, and in a natural way, the unit operators on  $ L:= \oplus_{i \in I}L_i$ and $E:= \oplus_{j \in J} E_j,$ and in the coming arguments $1_E, \, 1_L$ will sometimes denote a unit operator on a Hilbert space and sometimes a square block matrix with units in the main diagonal.  
 
 In the article \cite{Cs} we showed how the block Schur product has a natural representation as a completely bounded bilinear operator, such that for a Hilbert space $H$ and any pair $A, B $ of block matrices in $M_n(B(H))$  we have  $A \square B \, = V^* \, \l(A) F \l(B) V,$ where $\l$ is a $^*-$representation  of $M_n \otimes B(H)$ on the  Hilbert space $K$ defined by $K := \bc^n \otimes H \otimes \bc^n,$  $F$ is a self-adjoint unitary on $K$ and $V$ an isometry of $\bc^n \otimes H$ into $K.$ 
 
 The bounded matrices $X = (x_{ij})$ with $x_{ij} \in B(E_j, F_i)$ do not form an algebra in any obvious way, so the word {\em representation, } which was used above can not be used here, but it is still possible to define linear mappings $\l$ and $\r$ which play a similar roles in the block Schur tensor product as the ones $\l$ plays respectively to the left and to the right above.  
The mappings $ \l $ and $\r$ mimick the left and right regular representations known from discrete groups and they are defined below.

\begin{definition}
Let  $(E_j)_{j \in J}, \, (F_i)_{i \in I}, \,$ families of Hilbert spaces, then \newline $Mat((E_j, F_i))$ denotes all matrices $A = (A_{ij}) $   with $A_{ij} \in B(E_j, F_i).$  
\end{definition}

\begin{definition}
Let  $(E_j)_{j \in J}, \, (K_j)_{j \in J}, \, (F_i)_{i \in I}, \, (L_i)_{i \in I}$ be families of Hilbert spaces. 
\begin{itemize}
\item[(i)] $E:= \oplus_{j \in J} E_j,$ and $e_j$ is the orthogonal projection from $E$ onto $E_j.$
 
\item[(ii)] $F:= \oplus_{i \in I} F_i,$ and $f_i$ is the orthogonal projection from $F$ onto $F_i.$
 
 \item[(iii)] $K:= \oplus_{j \in J} K_j,$ and $k_j$ is the orthogonal projection from $K$ onto $K_j.$
 
 \item[(iv)] $L:= \oplus_{i \in I} L_i,$ and $l_i$ is the orthogonal projection from $L$ onto $L_i.$

\item[(v)] The linear mapping $\l$ of $Mat(E_j, F_i) $ to $Mat(E_j \otimes L, F_i \otimes L) $ is defined by \newline $\l\big((A_{ij} )\big) := (A_{ij} \otimes 1_L).$ 
\item[(vi)] The linear mapping $\r$ of $Mat(K_j, L_i) $ to $Mat(E \otimes K_j,  E \otimes , L_i ) $ is defined by \newline $\r\big((B_{ij} )\big) := (1_E \otimes B_{ij} ).$ 
\item[(vii)] The subspace $P$ of $E \otimes K$ is the closed linear space spanned by the family of pairwise orthogonal subspaces $E_j \otimes K_j,$ and the orthogonal projection of $E \otimes K$ onto $P$ is denoted $p.$ The isometry $v$ of $P$ into $E \otimes K$ is the natural embedding.  
\item[(viii)] The subspace $Q$ of $F \otimes L$ is the closed linear space spanned by the family of pairwise orthogonal subspaces $F_i \otimes L_i.$  The isometry $w$ of $Q$ into $F \otimes L$ is the natural embedding. 
\end{itemize} 
\end{definition}
 We remark without proof the following proposition.
 \begin{proposition} 
 \begin{itemize} 
 \item[(i)] The projection $p$ is given as the strongly convergent sum $p = \sum_{j \in J} e_j \otimes k_j $ in $B(E \otimes K).$ 
 
 \item[(ii)] The projection $q$ is given as the strongly convergent sum $q = \sum_{i \in I} f_i\otimes l_i $ in $B(F \otimes L).$  
 \end{itemize}
 \end{proposition}

In the rest of this section we will assume that the families of Hilbert spaces $ E_i, F_j, K_i, L_j$ are given and the notation just introduced in the definitions above will be used in this context.  The next lemma is the basis for the possibility to express the coming main theorem in a meaningful way even if the matrices involved are unbounded as ordinary operators. It should be remarked that if the index sets are all finite then the matrices involved will represent bounded operators.

\begin{lemma} \label{lemma}
\begin{itemize}
\item[]
\item[(i)] Let $A = (A_{ij} ) $ be a row bounded element in $Mat(E_j, F_i),$ then $q\l(A)$ is well defined as a bounded operator in $B\big((E \otimes K), (F \otimes L)\big)$ and $\|q\l(A)\| \leq  \|A\|_r.$ \newline
If all the spaces $F_i$ are non zero-dimensional then $\|q\l(A)\| = \|A\|_r.$ 
\item[(ii)] Let $B = (B_{ij} ) $ be a column  bounded element in $Mat(K_i, L_j),$ then $\r(B)p$ is well defined as a bounded operator in $B\big((E \otimes K), (F \otimes L)\big)$ and $\|\r(B)p\| \leq  \|B\|_c.$
 \newline
If all the spaces $E_j$ are non zero-dimensional then $\|\r(B)p\| = \|B\|_c.$ 
\end{itemize}

\end{lemma} 
\begin{proof}
We will only prove item (ii), since item (i)  follows from item (ii), applied to $A^*.$
 
Let $\Xi $ in $P$ be given as  $\Xi= (\xi_j)$ with $\xi_j \in E_j \otimes K_j,$ then for  given $i$ in $I$ and $j$ in $J$ we have $(1_E \otimes B_{ij}) \xi_j $ is in $E_j \otimes L_i$  and then for two indices $s, t$ in $J$ we have  $$ \langle(1_E \otimes B_{is})\xi_s, \,(1_E \otimes B_{it}) \xi_{t} \rangle \, = \,\begin{cases} 0 \, \text{ if } s \neq t \\
\|(1_E \otimes B_{is}) \xi_s\|^2 \, \text{ if } s =t. \end{cases} $$ 
We can then estimate the norm of $\r(B)p$ by the following majorizations 
\begin{align}
\|\r(B) \Xi\|^2 \,& =  \,
\sum_{i \in I}\| \sum_{j \in J} (1_E \otimes B_{ij} )\xi_j\|^2 \\ \notag & = \,\sum_{i \in I} \sum_{j \in J}\|(1_E \otimes B_{ij} )\xi_j\|^2 \\ 
\notag &= \,\sum_{j \in J} \sum_{i \in I}\| (1_E \otimes B_{ij} )\xi_j\|^2 \\ 
\label{diagelement} &= \,\sum_{j \in J} \langle \big( 1_E \otimes (\sum_{i \in I}B_{ij}^* B_{ij})\big)\xi_j, \xi_j \rangle \\ 
\notag & \leq \|B\|_c^2\sum_{j \in J} \|\xi_j\|^2 \\ 
\notag & = \|B\|_c^2 \|\Xi\|^2 ,
\end{align}
so $\|\r(B)p\| \, \leq \, \|B\|_c.$ To see the opposite inequality in the case when all the spaces $E_j$ are non zero,  we choose for each $j$ in $J$ a  unit vector $\a_j$ in $E_j$ and then  

\begin{align}
\|B\|^2_c \, & = \, \underset{j \in J, \, \b_j \in K_j, \, \|\b_j\|= 1}{\sup} \sum_{i \in I} \|B_{ij}\b_j\|^2 \\
\notag & = \, \underset{j \in J, \, \b_j \in K_j, \, \|\b_j\|= 1}{\sup} \sum_{i \in I} \|(1_E \otimes B_{ij}) (\a_j \otimes \b_j)\|^2 \\
\notag & = \, \underset{j \in J, \, \b_j \in K_j, \, \|\b_j\|= 1}{\sup} \|(1_E \otimes B) (\a_j \otimes \b_j)\|^2 \\
\notag & \leq \, \|\r(B)p\|^2,
\end{align} 
and the lemma follows.
\end{proof}

  Before we present the theorem, we will like to remark, that the  expression  $w^*\l(A)\r(B)v$  is  meaningful if the set of indices $I, J$ are both finite. If one or both sets of indices are infinite, but $A$ is row bounded and $B$ is column bounded, then, since  $vv^* = p$ and $ww^* = q,$ the lemma shows that   the expression \newline $(w^*\r(B))(\l(A)v)$ makes sense as a product of two bounded operators.

\begin{theorem} \label{ThBST} Let  $(E_j)_{j \in J}, \, (K_j)_{j \in J}, \, (F_i)_{i \in I}, \, (L_i)_{i \in I}$
be  families of Hilbert spaces, \newline $A= (A_{ij}) \in Mat(E_j,F_i)$ be a row bounded matrix and $B= (B_{ij}) \in Mat(K_j, L_i)$ be a column bounded matrix  then $A \boxtimes B$ is the matrix of the bounded operator \newline $\big(w^*\r(A)\big)\big(\l(B)v\big)$ and 
\begin{itemize} 
\item[(i)]  
 $\|A\boxtimes B\|_{op} \leq \|A\|_r \|B\|_c .$
 \item[(ii)] For any vector $\Xi \in \oplus_{j \in J} E_j \otimes K_j$ and any vector $\Gamma \in \oplus_{i \in I} F_i\otimes L_i: $
 \begin{align*}
 |\langle (A\boxtimes B) \Xi, \Gamma \rangle |& \,\leq \| \bigg(\big(\mathrm{diag}(AA^*)\big)^{\frac{1}{2}} \boxtimes 1_L \bigg)\Gamma\| \| \bigg(1_E \boxtimes \big(\mathrm{diag}(B^*B)\big)^{\frac{1}{2}}  \bigg)\Xi \| \\
&\, = \| \bigg(\big(\mathrm{diag}(AA^*)\big)^{\frac{1}{2}} \otimes 1_L \bigg)\Gamma\| \| \bigg(1_E \otimes \big(\mathrm{diag}(B^*B)\big)^{\frac{1}{2}}  \bigg)\Xi \|  
  \end{align*}
 \end{itemize} 

\end{theorem}

\begin{proof}
It follows directly from Lemma \ref{lemma} and the few lines following its proof, that \newline $\big(w^*\r(A)\big)\big(\l(B)v\big)$ is a bounded operator and the matrix of this operator has entries which are given as 
\begin{align}
 (f_i \otimes l_i)(A\otimes 1_L)(1_E \otimes B )| (E_j \otimes K_j) \, &= \, (f_i \otimes l_i)(A \otimes B)| (E_j \otimes K_j) \\
\notag  &= \, A_{ij}  \otimes B_{ij} | E_j \otimes K_j  \\
\notag &= ( A \boxtimes B)_{ij},  
 \end{align} 
 so the matrix of the bounded operator  $\big(w^*\r(A)\big)\big(\l(B)v\big)$ is $A \boxtimes B.$ We will identify a bounded operator with its matrix, when no confusion can occur and then we may write   $A \boxtimes B \, =\, \big(w^*\r(A)\big)\big(\l(B)v\big),$ so the result (i) follows from  Lemma \ref{lemma}.

For the proof of item (ii) we remark that for the given vectors $\Xi, \Gamma$ the decomposition of $A \boxtimes B $ gives 
\begin{align}
& \, \, \, \quad|\langle(A \boxtimes B) \Xi, \Gamma\rangle|^2 \\ \notag & = \, |\langle\r(B) \Xi, \l(A)^*\Gamma\rangle|^2\\
\notag & \le \, \|\r(B) \Xi\|^2  \| \l(A)^*\Gamma\rangle\|^2, \text{ by the equality } ( \ref{diagelement})  \\ 
\notag & = \,\bigg(\sum_{j \in J} \langle\big( 1_E \otimes  (\sum_{i \in I}B_{ij}^* B_{ij})\big)\xi_j, \xi_j\rangle\bigg)\bigg(\sum_{i \in I} \langle \big((\sum_{j \in J}A_{ij} A_{ij}^*)\otimes  1_L\big)\gamma_i, \gamma_i \rangle\bigg) \\
\notag & = \, \langle \big(1_E \otimes \mathrm{diag}(B^*B)\big) \Xi, \Xi \rangle \langle \big(\mathrm{diag}(AA^*) \otimes 1_L\big) \Gamma, \Gamma \rangle \\
\notag & =  \, \|\big(1_E \otimes \mathrm{diag}(B^*B)\big)^{\frac{1}{2}} \Xi\|^2 \|\big(\mathrm{diag}(AA^*)\otimes 1_L\big)^{\frac{1}{2}} \Gamma\|^2 \\
\notag & =  \, \|\bigg(1_E \otimes \big( \mathrm{diag}(B^*B)\big)^{\frac{1}{2}}\bigg) \Xi\|^2 \|\bigg(\big(\mathrm{diag}(AA^*)\big)^{\frac{1}{2}}\otimes 1_L\big)\bigg) \Gamma\|^2 \\
\notag & \quad \, \,\text { since } \Xi \in Q \text{ and } \Gamma \in P\\
\notag & =  \, \|\bigg(1_E \boxtimes \big( \mathrm{diag}(B^*B)\big)^{\frac{1}{2}}\bigg) \Xi\|^2 \| \|\bigg(\big(\mathrm{diag}(AA^*)\big)^{\frac{1}{2}}\boxtimes 1_L\big)\bigg)\Gamma\|^2 ,
\end{align}
and the theorem follows.
\end{proof}

\section{DECOMPOSITIONS OF THE BLOCK SCHUR AND BLOCK SCHUR TENSOR PRODUCT} 

The statement in item  (ii) of Theorem \ref{ThBST}, and the analogous result for block Schur products  $|\langle(A \square B) \Xi, \Gamma\rangle | \leq \| \big(\mathrm{diag}(AA^*)\big)^{\frac{1}{2} } \Gamma\| \| \big(\mathrm{diag}(B^*B)\big)^{\frac{1}{2} } \Xi\| $ from Theorem 2.9, (v)
of \cite{Cs}  tell us that a certain factorization of the block Schur tensor product and of the block Schur product must exist. We will show that this factorization may be computed in a simple  explicit form. We will deal with the block Schur product first, because it is notationally simpler than the block Schur tensor product. 

We will keep the notation which was introduced in section 4, so let $i$ be an index  in $I,$  and let $i(A)$ denote the $i$th row of $A ,$ i.e. $i(A) = (a_{ij})_{j \in J}. $ Since we assume that $A$ is row bounded we may, and will, just as well consider $i(A) $ as the matrix of  a bounded operator in $B\big( \ell^2(J, H), H\big).$ We will also let $i(A)$ denote this operator, and as a bounded operator it has a polar decomposition to the right $i(A) = |i(A)^*|i(V) , $  such that $|i(A)^*| = \big(\mathrm{diag}(AA^*)\big)^{\frac{1}{2}}_{ii}, $  and $i(V) = (v_{ij})_{j \in J} $ is a partial isometry in  $B\big( \ell^2(J, H), H\big).$  We may then combine all the row matrices $i(V)$ for $i $ in $I$ into a row bounded matrix $V =(v_{ij} ) = (i(V)_j)$ such that each row is the matrix of a partial isometry and we get 
 $$ A = \big(\mathrm{diag}(AA^*)\big)^{\frac{1}{2}}V.$$

In an analogous way, for a column bounded matrix $B = (b_{ij} ) $ and for each fixed index $j $ in $J,$ we can make a polar decomposition - to the left - of the $j'$th column $B(j)$ of $ B $ such that $B(j) = W(j)\big(\mathrm{diag}(B^*B)\big)^{\frac{1}{2}}_{jj} $ and the column matrix $(W(j)_i )_{i \in I}  $ is a partial isometry in $B(H, \ell^2(I,H)).$
 We may the collect the columns  $W(j) $ into a column  bounded matrix $W = (w_{ij}) = (W(j)_i) $ such that each column is a partial isometry and $$ B = W\big(\mathrm{diag}(B^*B)\big)^{\frac{1}{2}}.$$

\begin{theorem} \label{bSDec}
Let $H$ be a Hilbert space, $I, J$ sets of indices, $A = (a_{ij} )_{ij \in I \times J} $ and $B=(b_{ij})_{ij \in I \times J}$  matrices with entries from $B(H)$  such that the row norm $\|A\|_r$ and the column norm $\|B\|_c$ are both finite, then 
\begin{itemize} 
\item[(i)] There exist matrices $V = (v_{ij})$ and $W = (w_{ij}) $ such that each row in $ V$ and each column in $W$ is the matrix of a partial isometry and 
$$ A \square B =  \big(\mathrm{diag}(AA^*)\big)^{\frac{1}{2}}\big(V \square W\big ) \big(\mathrm{diag}(B^*B)\big)^{\frac{1}{2}}.  $$
\item[(ii)]
There exists a contraction  matrix $S = (s_{ij}) $  in $\mathrm{Mat}_{I \times J} \big(B(H)\big) $ such that $A\square B $ may be written as a  product of bounded matrices: 
$$ A \square B =  \big(\mathrm{diag}(AA^*)\big)^{\frac{1}{2}} S \big(\mathrm{diag}(B^*B)\big)^{\frac{1}{2}}.  $$
\end{itemize}
\end{theorem}
\begin{proof} 
The lines preceding the theorem explain the algebraic identity 
claimed in the statement (i). The second statement is a consequence of Livshits' inequality, \cite{Li} Theorem 1.8 or \cite{Cs} (1.7), because $ V$ has row norm  at most $1$  and $W$ has column norm at most $1, $ so the operator norm of $S:= V \square W $ is at most $1.$
\end{proof}

It turns out that the decomposition of the block Schur product of Theorem  \ref{bSDec} may be extended to cover the block Schur tensor product too.

\begin{theorem} \label{bStDec}
Let  $(E_j)_{j \in J}, \, (K_j)_{j \in J}, \, (F_i)_{i \in I}, \, (L_i)_{i \in I}$
be  families of Hilbert spaces, \newline $A= (A_{ij}) \in Mat(E_j,F_i)$ be a row bounded matrix and $B= (B_{ij}) \in Mat(K_j, L_i)$ be a column bounded matrix  then 
\begin{itemize} 
\item[(i)] There exist matrices $V = (v_{ij})$ with  $v_{ij} $ is in $B(E_j, F_i)$  and $W = (w_{ij}) $ with  $w_{ij} $ in $B(K_j, L_i)$ such that each row in $ V$ and each column in $W$ is the matrix of a partial isometry and 
$$ A \boxtimes  B =  \bigg(\big(\mathrm{diag}(AA^*)\big)^{\frac{1}{2}}\boxtimes 1_L\bigg)\bigg(V \boxtimes  W\bigg )\bigg(1_E \boxtimes \big(\mathrm{diag}(B^*B)\big)^{\frac{1}{2}}\bigg).  $$
\item[(ii)]
There exists a contraction  matrix $S = (s_{ij}) $ with $s_{ij}$ in $B\big((E_j \otimes K_j),(F_i \otimes L_i) \big) $ such that    $A\boxtimes B $ may be written as a  product of bounded matrices: 
$$ A \boxtimes  B =  \bigg(\big(\mathrm{diag}(AA^*)\big)^{\frac{1}{2}}\boxtimes 1_L\bigg)S\bigg( 1_E \boxtimes \big(\mathrm{diag}(B^*B)\big)^{\frac{1}{2}}\bigg).  $$
\end{itemize}
\end{theorem}
\begin{proof} 
With respect to the proof of item (i), we start by recalling the proof of Theorem  \ref{bSDec}. Then $$ A \, = \, \big(\mathrm{diag}(AA^*)\big)^{\frac{1}{2}}V \quad  \text{ and }\quad B \, = \, W \big(\mathrm{diag}  (BB^*)\big)^{\frac{1}{2}},$$
such that $V = (v_{ij}) $ is a matrix where all rows are partial isometries and  $W = (v_{ij}) $ is a matrix where all columns are partial isometries. 
 The operator $q\l(A)$ may then be factorized accordingly and we get 
 \begin{align} \label{left} 
\big( q\l(A)\big) \, &= \, q (A \otimes 1_L) 
\\ 
\notag & \quad\,  \,  \text{ since } q \text{ commutes with  the bd. op. } \big(\mathrm{diag}(AA^*)\big)^{\frac{1}{2}}\otimes 1_L \\ 
\notag & =\,  \bigg(\big(\mathrm{diag}(AA^*)\big)^{\frac{1}{2}}\otimes 1_L\bigg)q\bigg( V \otimes 1_L\bigg)\\
\notag  & =\,  \bigg(\big(\mathrm{diag}(AA^*)\big)^{\frac{1}{2}}\otimes 1_L\bigg)\bigg(q\l(V)\bigg).
\end{align} 

Similarly we get 

\begin{align} \label{right}
\big( \r(B)p \big) \, &= \,  (1_E \otimes B)p  \\
\notag & \quad \,  \,  \text{ since } p \text{ commutes with  the bd. op. } 1_E \otimes \big(\mathrm{diag}(B^*B)\big)^{\frac{1}{2}}\ \\ 
\notag & =\, \bigg( 1_E \otimes W \bigg))  p\bigg(1_E \otimes \big(\mathrm{diag}(B^*B)\big)^{\frac{1}{2}}\bigg)\\
\notag & =\, \bigg(\r(W)p\bigg) \bigg(1_E \otimes \big(\mathrm{diag}(B^*B)\big)^{\frac{1}{2}}\bigg).
\end{align}

By combination of Theorem \ref{bSDec}, the equalities (\ref{left}) and (\ref{right}) 
\begin{align}
A \boxtimes B \,& = \, \big(q\l(A)\big)\big(\r(B)p\big) \\
\notag & = \, q\bigg(\big(\mathrm{diag}(AA^*)\big)^{\frac{1}{2}}\otimes 1_L\bigg)q\bigg(\big(q\l(V)\big)\big(\r(W) p\big)\bigg)p\bigg(1_E \otimes \big(\mathrm{diag}(B^*B)\big)^{\frac{1}{2}}\bigg)p\\
\notag & = \, \bigg(\big(\mathrm{diag}(AA^*)\big)^{\frac{1}{2}}\boxtimes 1_L\bigg)\bigg(V \boxtimes W \bigg)\bigg(1_E \boxtimes \big(\mathrm{diag}(B^*B)\big)^{\frac{1}{2}}\bigg),
\end{align}

and since $S:= V\boxtimes W$ is a contraction by Theorem \ref{bStDec}, the theorem follows. 
\end{proof}

\section{ON THE SET  $R \circ C$ WITH $\|R\|_r \leq 1 $ AND  $\|C\|_c \leq 1.$} 
We find that it is very interesting to know more about which operators in the unit ball of $M_n\big(B(H)\big)$ that may be expressed as a block Schur product $R \square C$ with $R, C$ in $M_n\big((B(H)\big)\,$ such that $\|R\|_r \leq 1$ and $\|C\|_c \leq 1.$ For $n = \infty$ and $H= \bc$ we can show that the ultraweakly closed convex hull of these operators will not contain a positive multiple of the unit ball of $B\big(\ell^2(\bn)\big).$ In the first place we used the program {\em Maple} to indicate that the set of operators $R\circ C$ is a {\em thin } set, and then we saw that  the theory of random matrices may be used to give a proof of this impression. Since we have no explicit examples, but only probabilistic arguments to show the existence of certain operators, we have included our Maple code, which can provide examples up to the dimension $n = 50,$ on an ordinary lap-top.    

The argument, which shows that the  set $\{R \circ C \, : \, \|R\|_r \leq 1 , \, \|C\|_c \leq 1 \,\} $ is thin, is a kind of   backwards Hahn-Banach statement, which is formulated in the corollary following the theorem.  We recall that for a scalar $n \times n$ matrix $T,$ the norm $\|T\|_1$ is the first Schatten norm, which equals the norm of the functional $ M \to \mathrm{tr}(MT).$ The theory of random matrices  will for $n$ sufficiently big provide a $T$ such that $\|T\|_1= 1$ and the values  $|\mathrm{tr}\big((R\circ C)T\big)| $ are all small, and in this way show that the set is {\em thin.} 

We begin with a lemma, which gives an upper estimate over $R, C$  of $|\mathrm{tr}\big((R\circ C)T\big)|.$ 

\begin{lemma} \label{maxlemma}
Let $n$ be a natural number and $T =(t_{ij})$ a matrix in $M_n(\bc)$ then
\begin{itemize}
\item[(i)] \begin{align*}&\sup \{ |\mathrm{tr}\big((R\circ C)T\big)|\, : \, R, C \in M_n(\bc), \, \|R\|_r \leq 1, \|C\|_c \leq 1 \,\}\\& = \sup \big\{\sum_{i=1}^n \big(\sum_{j=1}^n |t_{ij}|^2 b_{ij}\big)^{(1/2)}\,:\, b_{ij} \geq 0,\, \, \, \sum_{i=1}^n b_{ij} \leq 1\,\big\}
\end{align*}
\item[(ii)] \begin{align*} &\sup \{ |\mathrm{tr}\big((R\circ C)T\big)|\, : \, R, C \in M_n(\bc), \, \|R\|_r \leq 1, \|C\|_c \leq 1 \,\}\\& \leq n \max \{|t_{ij}|\, : \, 1 \leq i,j \leq n\,\}.
\end{align*}
\end{itemize}
\end{lemma}
\begin{proof}
In item (i) we may for a fixed $C$ with $\|C\|_c \leq 1$ and varying $R$ with $\|R\|_r \leq 1$ use an elementary property of the Hilbert space inner product to get
\begin{align*}
 \underset{\|R\|_r \leq 1}{\sup}|\mathrm{tr}\big((R\circ C)T \big)|\, &= \,\underset{\|R\|_r \leq 1}{\sup}| \sum_{i=1}^n\big(\sum_{j=1}^nr_{ij}(c_{ij}t_{ji}) \big)| \\
 &\, = \sum_{i=1}^n\sqrt{\sum_{j=1}^n |c_{ij} |^2|t_{ji}|^2} 
 \end{align*}
 We will let $b_{ij} := |c_{ij}|^2$, then the matrix $B:=(b_{ij})$ is given by $b_{ij} \geq 0 $ and $\sum_{i=1}^n b_{ij} \leq 1,$ and item (i) follows.  
 
 For item (ii) we remark that the statement in item (i) implies that we get a bigger maximal value if we replace the matrix $T$ by the  matrix $M$ which is defined by $m_{ij}\,:= \, \max \{|t_{kl}|\, : \, 1 \leq k,l \leq n\,\}$ for all pairs $(i,j).$ We will then look at the maximization over $B= (b_{ij})$ in the case where all $m_{ij} =1.$ First we remark, that if for a given $j$ we have $\sum_{i} b_{ij} < 1,$ then the desired value will increase with growing $b_{jj}, $ so we may assume that $b_{jj} = 1 -\sum_{i \neq j}b_{ij}.$ 
 We are then left with the  exercise to maximize  the function $$ \sum_{i=1}^n \sqrt{ \sum_{j \neq i} b_{ij} + 1 - \sum_{k \neq i} b_{ki}}$$ over the set of $b_{ij}, $ for $i \neq j $ given by $b_{ij} \geq 0 $ and for any $j$ we have $\sum_{i \neq j} b_{ij} \leq 1.$ The function we optimize is concave since $\sqrt{x}$  is concave and the object function is the square root function composed with an affine function of the variables $\{b_{ij}\,:\, i \neq j\,\}, $ so the object function of the variables $\{b_{ij}\,:\, i \neq j\,\}$ is concave and a stationary point in the domain will be a point where the maximum is attained.  Partial differentiation with respect to $b_{ij} $ gives the partial derivative 
 $$\frac{1}{2}\big(\sum_{k \neq i } b_{ik} +1 - \sum_{l \neq i} b_{li}\big)^{-(1/2)} -  \frac{1}{2}\big(\sum_{k \neq j } b_{jk} +1 - \sum_{l \neq j} b_{lj}\big)^{-(1/2)}$$ We find that the point $b_{ij} = \frac{1}{n} $ for $ i \neq j$ is a stationary point and the - maximal value attained here is $n, $ so item (ii) follows. 
\end{proof}
The following theorem is based on several numerical experiments perform-ed via Maple's powerful solver. Afterwards we realized that the theory of random matrices may be used to explain the outcome of the experiments. Since the program Maple operates with random variables which are evenly distributed on the 199 integers from $-99$ to $99,$ our theorem below is also based on such random variables.

\begin{theorem} \label{random}
There exists a natural number $n_0$ such that for each natural number $n \geq n_0$ there exists a real symmetric $n \times n$  matrix $T = (T_{ij}) $ such that $\|T\|_1 = 1$ and $$\forall R, C  \in M_n(\bc):\quad |\mathrm{tr}\big((R\circ C)T\big)| \leq 3\|R\|_r\|C\|_c n^{-\frac{1}{2}}.$$  
\end{theorem}

\begin{proof}  
We will look at random real valued matrices, and since Maple uses a random generator which produces integers in the interval $[-99, 99],$ we will consider random matrices, named $N$, where the entries are independent, integer valued and evenly distributed random variables on these integers.  In order to get a symmetric matrix we define the symmetric matrix as the symmetric part $SN := \frac{1}{2}\big(N + Transpose(N)\big). $ The mean value of all the variables are $0,$ and  the second moment $E(N_{ij}^2)= 3300.$ We define $s := \sqrt{\frac{3300}{2}} \sim 40.62,$
and then  we can define a Wigner matrix $X := (X_{ij})$ by
\begin{equation}
X_{ij} :=  \frac{1}{s\sqrt{n}}SN_{ij} .
\end{equation} 
 We base our use of the theory of random matrices on the book \cite{AGZ}, and  we found that the book is freely available at the last authors homepage. We are in debt to our colleague Steen Thorbj{\o}rnsen, who is an expert on free probability and random matrices. He directed us to this very useful book, and helped us in verifying that our arguments, based on this theory are correct. 
We remark first that for any natural number $k$ 
 the moments $E(|X_{ij}|^k)$ are all finite, so the condition (2.1.1) of the book is fulfilled. For $i < j$ we find that the variables $X_{ij}$ are independent and identically distributed with  $E(X_{ij}) =0,$  $E(X_{ij}^2) = \frac{1}{n}$ and the variables $X_{ii} $ are independent identically distributed with $E(X_{ii}) =0,$ $E (X_{ii}^2 ) = \frac{2}{n},$ so we have, what the book names a {\em Wigner matrix } $(X_{ij}).$ Fortunately we do not have to get far into the book \cite{AGZ} to find the theorem we will use.  Theorem 2.1.1 of \cite{AGZ},  which goes back to Wigner, says that the empirical distribution of the eigenvalues of the symmetric matrix $ X$ converges, weakly in probability, to the standard semicircle distribution. The latter is given by the density $\sigma(x):= \frac{1}{2\pi}\sqrt{4 - x^2} $ on the interval $[-2,2].$ 
 The empirical distribution of the eigenvalues of $ X$ is based on the point measures $\delta_{\lambda_i}$ with mass 1 placed in the $n$ eigenvalues  $\lambda_1 \leq \lambda_2 \leq \dots \leq \lambda_n$ 
 of $X$ such that $$L_X:= \frac{1}{n}\sum_{i=1}^n \d_{\l_i},$$ and the meaning of weak convergence in probability is explained in the  lines following Theorem 2.1.1 of that book, and we quote, using the notation from the book,
 $$ \forall \e > 0\, \forall f \in C_b(\br): \,\, \underset{ n \to \infty}{\lim}\,P(|\langle L_X,f \rangle - \langle \sigma, f \rangle| > \e) = 0. $$

  We will continue with this notation and define a real valued continuous and bounded function $f(t)$ on $\br$ by 
 \begin{equation}
 f(t) := \begin{cases} 2 \, \, \text{ if } |t| > 2 \\ |t| \text{ if } |t| \leq 2 \end{cases}
\end{equation} 

Then there exists a natural number $n_0$ such that 
\begin{equation} \forall n \geq n_0: \quad
P(|\langle L_X, f \rangle - \langle \sigma, f \rangle| > 0.02) < 0.5. 
\end{equation} 

This means that for any $n \geq n_0$ there exists a real symmetric   $n \times n $ matrix $ S = (S_{ij})$ such that $|S_{ij}| \leq 99$ and the matrix  $V$ defined by  $V:= \frac{1}{s\sqrt{n}}S $ has an emperical distribution $L_V,$
which satisfies \begin{equation}
\langle L_V, f \rangle \, > \, \langle \sigma , f \rangle - 0.02 = \frac{2}{2\pi}\int_0^2t\sqrt{4 -t^2}dt - 0.02 = \frac{8}{3\pi} -0.02 > 0.82. 
\end{equation}  
  
By definition of $L_V$ and $f$ we have 
\begin{align}
\frac{1}{n}\|V\|_1\,& = \, \langle L_V, |t| \rangle \,\geq \, \langle L_V, f \rangle \, > \,0.82 \text{ so }\\ \label{norm1}
\|S\|_1\, &=\, s \sqrt{n} \|V\|_1 > s n^{(3/2)} \cdot 0.82 > 33.30\cdot n^{(3/2)}. 
\end{align}
We may then define $T:= (1/\|S\|_1)S,$ so $\|T\|_1 =1$ and by item (ii) of Lemma \ref{maxlemma} and inequality \ref{norm1} we get
\begin{align}
\notag&\forall R, C \in M_n(\bc), \|R\|_r \leq 1 , \, \|C\|_c \leq 1:\\ &\,  |\mathrm{tr}\big( (R \circ C) T\big)| \, \leq \, \frac{99n}{\|S\|_1} \leq \frac{99n}{33.3n^{(3/2)}} < 3n^{-(1/2)},
\end{align}
and the theorem follows. 
\end{proof}
 
 \begin{corollary}
 There exists a natural number $n_0$ such that for all natural numbers $n \geq n_0$ there exists a self-adjoint projection $E$ in $M_n(\bc)$ such that $3n^{-(1/2)}( 2E -I)$ is not an element of the closed convex hull $\cc$ of the set $\{R \circ C\, :\, \|R\|_r \leq 1\, \|C\|_c \leq 1\,\}.$  
 \end{corollary}
\begin{proof} 

By the theorem there exists for $n \geq n_0$ a self-adjoint operator $ T$ with $\|T\|_1 =1$ and $\max\{|\mathrm{ tr} (YT)| \, : Y \in \cc\,\} < 3n^{-(1/2)}.$ Let $E$ be the range projection for the positive part of $T$ then  $\mathrm{tr}( 3n^{-(1/2)}(2E-I)T ) = 3n^{-(1/2)}$ so $ 3n^{-(1/2)}(2E-I)$
is not a in $\cc.$ 
\end{proof} 
 
\medskip
\noindent
\large{\bf{{The Maple}} code}

As mentioned above we got the inspiration to the content of Theorem \ref{random} from experiments with Maple and random matrices. Below we present the Maple code, and we are very thankful to our colleague S{\o}ren Eilers who was very helpful in getting the program to work. He has in \cite{EJ} advocated  for the use of Maple as a source of inspiration for mathematical insights, and we are happy to tell, that this idea has been successful here.  

The user chooses $n,$ the dimension named {\em dim}, then the program chooses an $n \times n$  matrix  with random integer values in the set $[-99, 99],$  and then  takes its symmetric part, which is denoted  $sN.$ Then the first Schatten  norm of the symmetric matrix  $sN$ is computed and the operator $T$ is  defined as $T := (1/\|sN\|_1)sN. $ The maximization of $|\mathrm{tr}\big((R\circ C) T)| $ is then based on the result from Lemma \ref{maxlemma} item(i), which easily may be transformed into a few lines of Maple code. In the end, the program compares the maximal value times the square root of the dimension $n$  in the variable {\em MaxSqrtDim}. According to Theorem \ref{random} this value should be at most 3 in at least half of the experiments.  In my experiments it is always less than 2.5.

\medskip
\noindent
\begin{verbatim}
restart;
with(LinearAlgebra): with(Optimization):
dim:=5;
N:=RandomMatrix(dim,dim,datatype=float);
N:=1/2*(N+Transpose(N));
A:=SingularValues(N,output=[S]);
N1:=add(A[i],i=1..dim);
M:=(1/N1)*N;
object:=add(sqrt(add(b[i,j]*M[j,i]^2,j=1..dim)),i=1..dim);
boundary:={seq(add(b[i,j],i=1..dim)<=1,j=1..dim)};
midpoint:={seq(seq(b[i,j]=1/dim,i=1..dim),j=1..dim)};
MV:= Maximize(object,boundary,assume=nonnegative,       
     initialpoint=midpoint);                                                        
MaxvSqrtDim := MV[1]*sqrt(dim + 0.00000000001);
\end{verbatim}

\bigskip
We have made tests up to the dimension 50, and the variable {\em MaxSqrtDim } was, as told above,  always below 2.5, but for each dimension there was only a small variation of the size 0.1 in the answers for different random matrices. The value seems to be growing slowly with the dimension, such that for dimensions less than 10 it is  less than 2.1, whereas it is bigger than 2.4 for dimensions bigger than 40.

\section{FINAL REMARKS}

The decompositions of the block Schur product and the block Schur tensor product as presented in Theorem \ref{bSDec} and  
Theorem \ref{bStDec}  are based on the polar decomposition result for a single operator, and our decomposition do have something in common with that construction. On the other hand the middle term $S$ is not unique in an obvious way, whereas the operators $V \square W$ and  $V \boxtimes W$ are given by  explicit constructions, just as in the case of the polar decomposition. 

In the case of the polar decomposition say $A = U |A|$ we think of $|A| := (A^*A)^{\frac{1}{2}}$ as a numerical value, but as far as we know, there are no obvious analogies in bilinear algebra to  the roles played by the factors $\big(\mathrm{diag}(AA^*)\big)^{\frac{1}{2}}$
and  $\big(\mathrm{diag}(B^*B)\big)^{\frac{1}{2}}.$

The meaning of Theorem \ref{random}, is not fully digested, as we see it, and we have no way to give explicit examples, for $n$ large, of operators which may play the role of $T$ in  Theorem \ref{random}.

The numerical examples show some interesting patterns, which we have not been able to get into a theorem. It turns out, that for all the random matrices, Maple has produced for us and made computations with, the set of pairs $(i,j)$ in $\{1, \dots,n\} \times \{1, \dots,n\}$ with $ b_{ij} \neq 0$  never has more than $2n$ elements and for rather many of these pairs we have $b_{ij} =1.$

\end{document}